\documentclass{scrartcl}

\usepackage{graphicx}     
\usepackage{amsfonts, amssymb, latexsym, mathtools}
\usepackage{amsmath}
\usepackage{amsthm}
\usepackage{graphicx}
\usepackage{epstopdf}
\usepackage{nicefrac}

\usepackage{color}

\newcommand{\eBox}{$\hfill\square$}
\newcommand{\bu}{\mathbf{u}}
\newcommand{\bv}{\mathbf{v}}
\newcommand{\bq}{\mathbf{q}}

\newtheorem{defn}{Definition}
\newtheorem{exmp}[defn]{Example}

\newtheorem{prop}[defn]{Proposition}

\newtheorem{rem}[defn]{Remark}

\begin{document}

\title{Towards Velocity Turnpikes in \\ Optimal Control of Mechanical Systems\footnote{Timm Faulwasser, Sina Ober-Bl\"{o}baum, and Karl Worthmann are indebted to the German Research Foundation (DFG-grant WO\, 2056/4-1 and WO\, 2056/6-1). Moreover, Kathrin Fla\ss{}kamp, Sina Ober-Bl\"{o}baum, and Karl Worthmann gratefully acknowledge funding by \textit{Mathematisches Forschungsinstitut Oberwolfach}.}}

\author{Timm Faulwasser\footnote{T.~Faulwasser,
			Institute for Automation and Applied Informatics,
			Karlsruhe Institute of Technology (KIT), Germany,
			e-mail: timm.faulwasser@ieee.org}
\and
Kathrin Fla\ss{}kamp\footnote{K. Fla{\ss}kamp, 
              Center for Industrial Mathematics,
              University of Bremen, Germany,
              email: kathrin.flasskamp@uni-bremen.de}
\and
Sina Ober-Bl\"{o}baum\footnote{S. Ober-Bl\"obaum,
             Department of Engineering Science,
             University of Oxford, United Kingdom,
             email: sina.ober-blobaum@eng.ox.ac.uk}
\and
Karl Worthmann\footnote{K. Worthmann,
          Institut f\"{u}r Mathematik,
         Technische Universit\"{a}t Ilmenau, Germany,
        email: karl.worthmann@tu-ilmenau.de}}

\maketitle

\begin{abstract}
	The paper proposes first steps towards the formalization and characterization of time-varying turnpikes in optimal control of mechanical systems. %
	We propose the concepts of velocity steady states, which can be considered as partial steady states, and hyperbolic velocity turnpike properties for analysis and control. %
	We show for a specific example that, for all finite horizons, %
	both the (essential part of the) optimal solution and the orbit of the time-varying turnpike correspond to (optimal) trim solutions. %
	Hereby, the present paper appears to be the first to combine the concepts of trim primitives and time-varying turnpike properties.
\end{abstract}


\section{Introduction}
Nowadays optimal and predictive control are established methods that are applied to a wide range of control problems in mechanics, mechatronics and robotics. After the seminal conceptual breakthroughs around the mid of the 20th century the  main driving force of this trend has been the development of powerful numerical methods. 

Just recently the analysis of parametric Optimal Control Problems (OCPs) --~e.g.\, the ones arising in model predictive control~-- has seen renewed interest in the concept of turnpike properties, cf.\, \cite{Stieler14a,Trelat15a,epfl:faulwasser15h,kit:faulwasser18c}. 
Turnpikes are a classical topic, originating from the analysis of problems arising in economics and subsequently found to be ubiquitous in many application areas of optimal control. 
It refers to a similarity property of parametric optimal control problems, whereby the boundary conditions of the dynamics and the horizon length are varied. Originally, the term turnpike has been coined by \cite{Dorfman58} and popularized by \cite{Mckenzie76} and \cite{Carlson91}; early reports of turnpike phenomena can be traced back to \cite{vonNeumann38}.

The majority of recent works on turnpikes in optimal control focuses on steady-state turnpikes \cite{Trelat15a,Gugat16,epfl:faulwasser15h}, i.e.\, on problems where the turnpike can be understood as the steady-state attractor of infinite-horizon optimal solutions. However, it has been well understood in economics that turnpikes can as well be time-varying orbits \cite{Samuelson76, Zaslavski06}. %
Recently, non-periodic time-varying turnpikes in general OCPs have been studied by~\cite{Gruene17b}. %
While steady-state turnpikes are elegant in the sense that they are obtained by computing the optimal steady state of the system --~i.e.\, by solving a simple NLP~-- so far it is not clear how to compute a non-periodic time-varying turnpike orbit directly without solving a large number of OCPs. 

In the present paper, we aim at characterizing non-periodic turnpike orbits for a class of OCPs arising in mechanical systems. Specifically, we exploit symmetries and the concept of trim primitives \cite{Frazzoli2001,FrazBull02}. %
The symmetries we consider can be represented by Lie groups which induce invariances, i.e.\, the system dynamics are invariant w.r.t.\, the corresponding symmetry actions. %
For mechanical systems with translational or rotational symmetries, this means that translations or rotations of a trajectory lead to another trajectory of the control system. Hence, a solution trajectory that has been designed for one specific situation can be (re-)used for another situation as well which defines an equivalence class of solutions whose representative is called a primitive. %
Particularly, induced by symmetries there may exist trim primitives or trims, for short, (see~\cite{FrDaFe05}) which are basic motions, e.g.\, going straight at constant speed or turning with constant rotational velocity in mechanical systems. %
Trims can be represented very conveniently with Lie group actions, even if general solutions of the dynamical systems cannot be computed by hand. So far, the concept of (trim) primitives has been widely used in motion planning for (hybrid) dynamical systems 
\cite{FrDaFe05,FOK12,FHO15}. 
In principle, primitives are used to build up a library of solutions for intermediate optimal control problems. Dependent on the specific control scenario, an optimal path can be searched for in this motion library very quickly. This allows for solving optimal control problems very effectively online.
However, the relation of trim primitives to turnpikes has not been established yet.

Summing up, the main contribution of the present paper are first steps towards the formalization and characterization of time-varying turnpikes for a rather broad class of OCPs arising in control of mechanical systems. To this end, we propose the concept of \emph{hyperbolic velocity turnpike properties}, which are slightly more general than exponential turnpikes and less general than their measure-based counterparts.\footnote{To see this, observe that any exponential bound $C(e^{-\gamma t}+e^{-\gamma (T-t)})$ can for all $t \in [0,T]$ be bounded from above by a hyperbolic function $\frac{\tilde C}{T}$, which is independent of $t$. Moreover, our proposed definition of hyperbolic turnpikes directly implies the measure-based variant. A detailed investigation of these relation is subject to future work.} We show for a specific example that for all finite horizons the optimal solutions can be characterized by a specific sequence of trim solutions and that the time-varying turnpike orbit corresponds to an optimal trim solution. 
To the best of the authors' knowledge the present paper appears to be the first one to explicitly combine the concepts of trim primitives and time-varying turnpike properties. 
 
The remainder of the paper is structured as follows. Section~\ref{sec:velo_steady} introduces the concept of velocity steady states, provides background on symmetries in mechanical systems and introduces the problem at hand. Section~\ref{sec:Exmpl1D} draws upon a motivational example to illustrate the concept of a velocity turnpike, while Section~\ref{sec:nonlinMech} presents numerical results for a nonlinear mechanical example. Finally this paper ends with conclusions and outlook in Section~\ref{sec:conclusions}. 

\textbf{Notation}: $\\mathcal{L}^\infty([0,T],\mathbb{R}^m)$, $m \in \mathbb{N}$ and $T \in \mathbb{R}_{>0}$, denotes the space of Lebesgue-measurable and absolutely integrable functions $f:[0,T] \rightarrow \mathbb{R}^n$. $\| \cdot \|$ denotes the Euclidean norm. 

\section{Velocity Turnpikes and Trim Primitives} \label{sec:velo_steady}

In this section we define a velocity steady state, which forms the basis for the concept of a velocity turnpike. Both concepts turn out to be the suitable generalization of the terms steady state and turnpike in optimal control of mechanical systems. Before doing so, we briefly recap mechanical systems with a particular focus to symmetries.

\subsection{Mechanics and Symmetry}\label{sec:MechSymmetry}


Let $Q$ denote the $n$-dimensional smooth manifold of configurations and the system dynamics be given by Euler-Lagrange equations
\begin{align}\label{NotationSystemDynamics}
		\frac{\mathrm{d}}{\mathrm{d}t} \frac{\partial L}{\partial \dot{\mathbf{q}}} - \frac{\partial L}{\partial \mathbf{q}} & = f_L(\mathbf{q},\dot{\mathbf{{q}}},\mathbf{u})
\end{align}
with real-valued Lagrangian $L$ and mechanical forces $f_L$ depending on external controls $\mathbf{u} \in \mathbb{R}^m$ where $\mathbf{q}\in Q$ and the state space is given by the tangent bundle $TQ$. Assuming regularity of the Lagrangian, the second-order Euler-Lagrange equations can be reformulated as a system of first-order Ordinary Differential Equations (ODEs) in the form $\dot{\mathbf{x}} = f(\mathbf{x},\mathbf{u})$ where $\mathbf{x} = (\mathbf{q},\dot{\mathbf{q}}) = (\mathbf{q},{\mathbf{v}}) \in T_q Q$ denotes the full state, which is contained in the tangent space at~$q$. %
Then, the solution~$\mathbf{x}(t) = \phi_u(t;\mathbf{x}_0)$ to the Euler-Lagrange Eq.~\eqref{NotationSystemDynamics} for initial condition $\mathbf{x}_0$ is given by the (forced Lagrangian flow) $\phi_u: [0,T] \times TQ \to TQ$ for $u \in \mathcal{L}^1([0,T],\mathbb{R}^m)$.

A Lie group is a group $(\mathcal{G},\circ)$, which is also a smooth manifold, for which the group operations $(g,h) \mapsto g \circ h$ and $g \mapsto g^{-1}$ are smooth. If, in addition, a smooth manifold~$M$ is given, we call a map $\Psi: \mathcal{G} \times M \rightarrow M$ a left-action of $\mathcal{G}$ on $M$ if and only if the following properties hold:
\begin{itemize}
	\item $\Psi(e,\mathbf{x}) = \mathbf{x}$ for all $\mathbf{x} \in M$ where $e$ denotes the neutral element of $(\mathcal{G},\circ)$,
	\item $\Psi(g,\Psi(h,\mathbf{x})) = \Psi(g \circ h,\mathbf{x})$ for all $g, h \in \mathcal{G}$ and $\mathbf{x} \in M$.
\end{itemize}

\begin{defn}[Symmetry Group]\label{def:symmetryGroup}
	Let the configuration manifold $Q$ be a smooth manifold, $(\mathcal{G},\circ)$ a Lie-group, and $\Psi$ a left-action of $\mathcal{G}$ on $Q$. Further, let $\Psi^{TQ}: \mathcal{G} \times TQ \to TQ$ be the lift of $\Psi$ to $TQ$. Then, we call the triple $(\mathcal{G},Q,\Psi^{TQ})$ a \emph{symmetry group} of the system \eqref{NotationSystemDynamics} if the property
\begin{align}\label{NotationCommutativityFlowSymmetry}
	\phi_u(t;\Psi^{TQ}(g,\mathbf{x}_0)) = \Psi^{TQ}(g,\phi_u(t;\mathbf{x}_0)) \quad\forall\, t \in [0,T]
\end{align}
holds for all $(g,\mathbf{x}_0,u) \in \mathcal{G} \times TQ \times \mathcal{L}^1([0,T],\mathbb{R}^m)$.
\eBox
\end{defn}
\begin{defn}[Trim Primitive] \label{def:Trims}
Let $(\mathcal{G},Q,\Psi^{TQ})$ be a symmetry group in the sense of 
	Definition~\ref{def:symmetryGroup}.
	Then, a trajectory $\phi_u(\cdot;\mathbf{x}_0)$, $\mathbf{u}(t) \equiv \bar{\mathbf{u}} = \text{const.}$, is called a \emph{(trim) primitive} if there exists a Lie algebra element $\xi \in\mathfrak{g}$ such that
	\[ 
		\phi_u(t;\mathbf{x}_0) = \Psi^{TQ}(\exp(\xi t),\mathbf{x}_0) \quad\forall\, t \geq 0.\qquad\text{\eBox}
	\]
\end{defn}
For a formal definition of Lie algebras we refer to~\cite{Bake12}. Instead, we illustrate the introduced concepts by means of the following example.
\begin{exmp}
Consider the mechanical system of the particular form
		\begin{equation}\label{NotationSecondOrderSystem}							\begin{aligned}
			\dot{\mathbf{q}}(t) & = \mathbf{v}(t) \\
			\dot{\mathbf{v}}(t) & = f(\mathbf{v}(t),\mathbf{u}(t))
			\end{aligned}
		\end{equation}
				This system class is invariant w.r.t.\ translations in $q$.
		A trim can be characterized by the pair $(\bar{\mathbf{v}},\bar{\mathbf{u}})^\top$ satisfying the condition $f(\bar{\mathbf{v}},\bar{\mathbf{u}}) = 0$. Then, we get the solution trajectories $\mathbf{q}(t) = \mathbf{q}_0 + \bar{\mathbf{v}}t$ and $\mathbf{v}(t) = \mathbf{v}_0 = \bar{\mathbf{v}}$. This can also be expressed via
		\begin{align*}
			\Psi^{TQ} \left(\exp(\xi t), \begin{pmatrix}
				\mathbf{q}_0 \\ \mathbf{v}_0
			\end{pmatrix} \right) & 
			= \begin{pmatrix}
				\mathbf{q}_0 + \boldsymbol{\xi} t \\ \mathbf{v}_0
			\end{pmatrix}. \qquad\text{\eBox}
		\end{align*}
\end{exmp}

\subsection{Velocity Steady States and Velocity Turnpikes}

Let the stage cost $\ell: \mathbb{R}^{n} \times \mathbb{R}^{m} \to \mathbb{R}$ be  continuous and convex and let the closed sets $\mathbb{U} \subseteq \mathbb{R}^m$ and $\mathbb{X} \subseteq \mathbb{R}^n$ be given. %
We consider the OCP
\begin{align}
	\underset{u \in \mathcal{L}^1([0,T],\mathbb{R}^{m})}{\text{minimize}} \quad & \int_{0}^{T} \ell(\mathbf{x}(t),\bu(t))\,\mathrm{d}t \nonumber \\
	\text{subject to}& \label{eq:OCP}\\
		 \dot{\mathbf{x}}(t) &= f(\mathbf{x}(t),\bu(t)) \quad \forall\, t \in [0,T]\nonumber\\
		 		 \mathbf{x}(0) &= \mathbf{x}_0 \text{ and } \mathbf{x}(T) = \mathbf{x}_{T} \nonumber\\
		 \bu(t) &\in \mathbb{U} \text{ and } \mathbf{x}(t) \in \mathbb{X} \quad \forall\, t \in [0,T] \nonumber
\end{align}
%
where the last three conditions refer to the system dynamics, the boundary conditions, and the control and state constraints.

A state~$\bar{\mathbf{x}} \in \mathbb{X}$, is called (controlled) equilibrium if there exists $\bar\bu \in \mathbb{R}^m$ such that $f(\bar{\mathbf{x}},\bar\bu) = 0$ holds. %
Based on this terminology, 	the pair $(\bar{\mathbf{x}},\bar\bu) \in \mathbb{R}^{n} \times \mathbb{R}^{m}$ is called an \emph{optimal steady state} if it holds that
\begin{align}\label{eq:SOP}
	(\bar{\mathbf{x}},\bar\bu) = \operatorname{argmin} \{ \ell(\mathbf{x},\bu) | (\mathbf{x},\bu) \in \mathbb{X}\times\mathbb{U}, f(\mathbf{x},\bu) = 0 \}\hspace*{-0.035cm}.
\end{align}

%


Classically, turnpikes are optimal steady states, i.e. solutions to \eqref{eq:SOP}. 
For mechanical systems which are modeled by second-order dynamics (see Section~\ref{sec:MechSymmetry}), a steady state always corresponds to zero velocity. %
However, as we have seen above, symmetries of mechanical systems may lead to trim trajectories, which have constant nonzero velocity and linear behavior in the configuration variables~$\mathbf{q}$. %
As we elaborate in the following, trims play an important role in optimal control of mechanical systems. %
In fact, we show that mechanical systems with symmetries can be optimally controlled on trims and these trims can be seen as time-varying or velocity turnpikes. %
Therefore, we extend the definition of steady states, to \emph{velocity steady states}, %
which have constant velocity~$v$, but dynamical motions in the configurations $q$.
\begin{defn}[Velocity Steady State]~\\
	$(\bar\bv, \bar\bu) \in \mathbb{R}^n \times \mathbb{R}^m$ is called a \emph{velocity steady state} for the mechanical control system
	\begin{equation} \label{eq:sys_mechanics}
		\frac{\mathrm{d}}{\mathrm{d}t} \begin{pmatrix} \bq\\ \bv \end{pmatrix} = \begin{pmatrix} \bv \\ f(\bv,\bu) \end{pmatrix}
	\end{equation}
	if $f(\bar\bv, \bar\bu) = 0$ holds.
	 \eBox
\end{defn}
On a velocity steady state, from $\bv(t) \equiv \bar\bv$ we directly see that $\mathbf{q}(t) = \bq_{0} +\bar\bv t$ holds for $t\geq0$, i.e.\ the position is (at least for $\bar{\mathbf{v}} \neq 0$) constantly changing, while the system remains in the \textit{velocity} steady state. %
\begin{rem}[Partial stability and velocity steady states]\label{RemarkVorotnikov}~\\
	It is worth to be noted that the notion of velocity steady state as defined above is a special case of the concept of \emph{partial steady states}. We refer to \cite{Vorotnikov12} for details on partial stability and partial steady states. Here, however, we prefer to focus on velocity steady states due to their close relation to trim primitives. 
	\eBox
\end{rem}
The definition of a \textit{velocity steady state} exploits the translational invariance of System~\eqref{eq:sys_mechanics}. Indeed, the more general definition would be a \textit{symmetry steady state} --~which would then also generalize the concept of a partial steady state as mentioned in Remark~\ref{RemarkVorotnikov}. In other words, a classical steady state refers to an affine subspace of dimension zero, a velocity steady state corresponds to an $n$-dimensional affine subspace, while the symmetry steady state reduces the remaining freedom of the underlying mechanical system~\eqref{NotationSystemDynamics} to \textit{motions} on a symmetry-induced manifold. Here, homogeneous coordinates are required to match the respective manifold to an affine subspace, see the explanations on representation of mechanical systems by~\cite{FlasOber19}.

Subsequently, we consider OCP~\eqref{eq:OCP} for systems~\eqref{eq:sys_mechanics}, i.e.\ the boundary constraints in~\eqref{eq:OCP} are given by $\mathbf{x}_0 = (\bq_0, \bv_0)$ and $\mathbf{x}_T = (\bq_T, \bv_T)$. Moreover, we restrict the initial conditions to a compact set $\mathbb{X}_0\subset\mathbb{X}$.\footnote{If one considers $\mathbb{X}_0 = \mathbb{X}$, then the subsequent turnpike definitions imply that $\mathbb{X}$ has to be controlled forward invariant, which might be overly restrictive. Hence, we restrict the set of initial conditions to an appropriate subset $\mathbb{X}_0\subset\mathbb{X}$.}

Next we propose a definition of a time-varying turnpike property, where the turnpike as such is a velocity steady state. Similarly to \cite{Carlson91} consider 
\begin{equation} \label{eq:Theta}
	\Theta_{T}(\varepsilon) = \left\{t \in [0,T]: \left\| (\bv^\star(t), \bu^\star(t)) - (\bar\bv, \bar\bu)\right\| > \varepsilon\right\},
\end{equation}
which is the set of time instances for which the optimal velocity and input trajectory pairs are not inside an $\varepsilon$-ball of the steady-state pair $(\bar\bv, \bar\bu)$. Now we are ready to define a measure-based velocity turnpike property similar to \cite{epfl:faulwasser15h}. 
\begin{defn}[Velocity turnpike property] \label{def:turnpike}~\\
	The optimal solutions $(q^\star,v^\star,u^\star)$ of OCP \eqref{eq:OCP} are said to have a \textit{velocity turnpike} w.r.t.\ $(\bar\bv, \bar\bu)$ %
	if there exists a function $\nu:\mathbb{R}_{\geq 0} \to \mathbb{R}_{\geq 0}$ such that, for all $\mathbf{x}_0 \in \mathbb{X}_0$ and all $T>0$, we have
	\begin{equation}\label{eq:TP}
		\mu\left[\Theta_{T}(\varepsilon)\right]  <\nu(\varepsilon)< \infty \quad \forall\: \varepsilon >0,
	\end{equation}
	where $\mu$ is the Lebesgue measure.
%
	\eBox
\end{defn}
As already mentioned, there exist alternative definitions of turnpike properties, %
see \cite{Stieler14a, Trelat15a} for so-called \textit{exponential} turnpikes and \cite{Gugat18a} for \emph{averaged (input)} turnpikes. %
For the purpose of this paper, we are interested in a slightly more general property, where the \textit{exponential bound} from  \cite{Stieler14a, Trelat15a}  is replaced by a \textit{hyperbolic function}. Note that every hyperbolic velocity turnpike is also a velocity turnpike.
\begin{defn}[Hyperbolic velocity turnpike property] \label{def:hyperbolic_velocity_turnpike}~\\
	The optimal solutions $(q^\star, v^\star, u^\star)$  of OCP \eqref{eq:OCP} are said to have a \textit{hyperbolic velocity turnpike} w.r.t.\, $(\bar\bv, \bar\bu)$ %
	if there exist positive constants $C, \tau_0(\bv_0), \tau_T(\bv_T)$  such that, for all $\mathbf{x}_0 \in \mathbb{X}_0$ and all sufficiently large $T > \tau_0(\bv_0) + \tau_T(\bv_T)\geq 0$, we have
	\begin{equation}\label{eq:hyp_vTP}
		\left\| (\bv^\star(t), \bu^\star(t)) - (\bar\bv, \bar\bu)\right\| \leq\frac{C}{T}
	\end{equation}
	for all $t\in [\tau_0(\bv_0),T-\tau_T(\bv_T)]$.\eBox
\end{defn}

Note the restriction of the optimization interval $[0,T]$ to $[\tau_0(\bv_0),\,T-\tau_T(\bv_T)]$ in the above definition, which is needed to allow for cases where the boundary velocities $\bv_0$ and $\bv_T$ are not close to $\bar\bv$. 
Interestingly, hyperbolic velocity turnpikes are closely related to \emph{averaged} (\emph{velocity}) turnpikes, who can be defined using ideas on  averaged (input) turnpike properties in PDE-constrained OCPs, see \cite{Gugat18a}. A detailed investigation of this connection is beyond the scope of this paper.

\section{Illustrative Example}\label{sec:Exmpl1D}
We consider the second-order system $\ddot{x}(t) = u(t)$. %
Firstly, we rewrite the system dynamics as a first-order ODE, i.e.
\begin{equation}\label{MotivationalExampleFirstOrderODE}
	\frac {\mathrm{d}}{\mathrm{d}t} \begin{pmatrix}
		q(t) \\ v(t)
	\end{pmatrix} = \begin{pmatrix}
		0 & 1 \\ 0 & 0
	\end{pmatrix} \begin{pmatrix}
		q(t) \\ v(t)
	\end{pmatrix} + \begin{pmatrix}
		0 \\ 1
	\end{pmatrix} u(t).
\end{equation}
The stage cost is given by $\ell(q,v,u) := \frac 12 (\Vert v \Vert^2 $ $+ \Vert u \Vert^2)$. If we impose, in addition, the boundary conditions
\begin{equation}\label{MotivationalExampleBoundaryConditions}
	\begin{pmatrix}
		q(0) \\ v(0)
	\end{pmatrix} = \begin{pmatrix}
		q_0 \\ v_0
	\end{pmatrix} \quad\text{ and }\quad \begin{pmatrix}
		q(T) \\ v(T)
	\end{pmatrix} = \begin{pmatrix}
		q_T \\ v_T
	\end{pmatrix},
\end{equation}
we get the OCP
\begin{align} 
	\underset{u \in \mathcal{L}^1([0,T],\mathbb{R})}{\text{minimize}}\quad & \int_{0}^{T} \frac 12 \left( \Vert v(t) \Vert^{2} + \Vert u(t) \Vert^{2} \right) \mathrm{d}t \label{eq:example_OCP} \\
			\text{subject to}\quad & \eqref{MotivationalExampleFirstOrderODE} \text{ for almost all $t \in [0,T]$ and \eqref{MotivationalExampleBoundaryConditions}}. \nonumber
	\end{align}

OCP~\eqref{eq:example_OCP} considers a controllable linear time-invariant system without input constraints and a stage cost $\ell$ strictly convex in $u$. Hence, for $T > 0$, classical results  can be used to show unique existence of an optimal solution.\footnote{More precisely, one can employ \cite[Thm. 8, p. 208]{Lee67} to show that the augmented reachable set corresponding to \eqref{eq:example_OCP} is a closed and convex subset of $\mathbb{R}^{n+1}$. Existence and uniqueness of $u^\star$ follows via geometric arguments \cite[p. 217]{Lee67}.}

Since the system is invariant w.r.t.\ translations in $q$, %
any triple $(q,v,u)$ with $(q,0,0)$ is a velocity steady state and all of them satisfy $\ell(q,v,u) = 0$, %
so the system is optimally operated at all of these steady states.

The symmetry group of the system~\eqref{MotivationalExampleFirstOrderODE} is $\mathcal{G} = \mathbb{R}$ with $\Psi_{g}(q) = q+g$. %
For the full state vector $(q,v)$, the symmetry action can be lifted to $\Psi^{TQ}_{g}(q,v) = (q+g,v)$. %
The Lie algebra is $\mathfrak{g} = \mathbb{R}$ and the exponential map is the identity.
Thus, for $\xi \in \mathfrak{g}$ and $v_0 = \xi$, trims are given by 
\[ 
	\begin{pmatrix}
		q(t) \\ v(t)
	\end{pmatrix} = \begin{pmatrix}
		q(0) \\ v(0)
	\end{pmatrix} + t \begin{pmatrix}
		\xi \\ 0
	\end{pmatrix} \qquad\text{for $u(t) \equiv 0$}. 
\]

\subsection{Pontryagin's Maximum Principle} 

Next, we apply Pontryagin's Maximum Principle, i.e.\ the necessary optimality conditions, to further analyze the example. To this end, we first define the Hamiltonian 
\begin{multline*}
\mathcal{H}(q,v,\lambda,u) = \\
	\lambda_{0} \left( \frac{1}{2} \Big( v^{2} + u^{2} \Big) \right) + 
	\underbrace{\lambda^\top \left( \begin{pmatrix}
		0 & 1 \\ 0 & 0
	\end{pmatrix} \begin{pmatrix}
		q \\ v
	\end{pmatrix} + \begin{pmatrix}
		0 \\ 1
	\end{pmatrix} u \right)}_{= \lambda_1 v + \lambda_2 u}.
\end{multline*}
Before we proceed, let us briefly show that OCP \eqref{eq:example_OCP} is normal, i.e. the multiplier $\lambda_0$ is not equal to zero (which allows dropping~$\lambda_0$ as an argument of the Hamiltonian~$\mathcal{H}$). %
Suppose that $\lambda_0 = 0$ holds. Then, $\mathcal{H}_u = \lambda_2$ holds, %
which implies $\lambda_2(t) \equiv 0$ and, thus, $\dot{\lambda}_2(t) = 0$. %
Plugging this condition into the equation $\dot{\lambda}_2 = - \mathcal{H}_v = - \lambda_1(t)$, %
yields $\lambda_1(t) \equiv 0$, i.e.\, a contradiction to the nontriviality of the multipliers. In conclusion, we can set $\lambda_0 := 1$ w.l.o.g. in the following.

The adjoint equations are
\begin{align*}
	\dot{\lambda}_1(t) & =  - \mathcal{H}_q(q^\star(t),v^\star(t),\lambda(t),u^\star(t)) = 0,  \\
	\dot{\lambda}_2(t) & =  - \mathcal{H}_v(q^\star(t),v^\star(t),\lambda(t),u^\star(t)) = -v^\star(t) - \lambda_1(t). 
\end{align*}
In addition, the maximum principle yields
\begin{equation*}
	\mathcal{H}_u(q^\star(t),v^\star(t),\lambda(t),u^\star(t)) = 0 \quad\Longleftrightarrow\quad u^\star(t) = -\lambda_2(t) 
\end{equation*}
for almost all $t \in [0,T]$, which can be used to eliminate the control from the optimality system (state-adjoint system) %
with Hamiltonian matrix $H_{OCP}$
\begin{equation}\label{eq:TPBVP}
	\frac {\mathrm{d}}{\mathrm{d}t} \begin{pmatrix} 
		q^\star(t) \\ v^\star(t) \\ \lambda_{1}(t) \\ \lambda_{2}(t)
	\end{pmatrix} = \underbrace{\begin{pmatrix} 
		0 & \phantom{-}1 & \phantom{-}0 & \phantom{-}0 \\ 0 & \phantom{-}0 & \phantom{-}0 & -1 \\ 0 & \phantom{-}0 & \phantom{-}0 & \phantom{-}0 \\ 0 & -1 & -1 & \phantom{-}0 
	\end{pmatrix}}_{=:H_{OCP}} \begin{pmatrix} 
		q^\star(t) \\ v^\star(t) \\ \lambda_{1}(t) \\ \lambda_{2}(t) 
	\end{pmatrix}
	\text{ s.t. \eqref{MotivationalExampleBoundaryConditions}.}
\end{equation}

Next, we solve the two-point boundary problem \eqref{eq:TPBVP}. To this end, we first compute
\begin{equation}\nonumber
	\det(\sigma- H_{OCP}) = \left| \begin{array}{rrrr}
		\sigma & -1 & 0 & 0 \\
		0 & \sigma & 0 & 1 \\
		0 & 0 & \sigma & 0 \\
		0 & 1 & 1 & \sigma
	\end{array} \right| = \sigma^2 ( \sigma^2 - 1 ) \stackrel{!}{=} 0,
\end{equation}
i.e.\ we get the eigenvalue $\sigma_1 := 0$ with algebraic multiplicity two and geometric multiplicity one. %
Hence, we get the eigenvector $w_{\sigma_1} = (1\ 0\ 0\ 0)^\top$ and the generalized eigenvector $h_{\sigma_1} = (0\ 1\ -1\ 0)^\top$. %
Moreover, we obtain the eigenvalues $\sigma_2 := 1$ and $\sigma_3 := -1$ %
with eigenvectors $w_{\sigma_2} = (1\ 1\ 0\ -1)^\top$ and $w_{\sigma_3} = (1\ -1\ 0\ -1)^\top$. 

Using these preliminary considerations allows us to compute the Jordan canonical form
\begin{equation}\nonumber
	\left( \begin{array}{rr|rr}
		0 & 1 & & \\
		0 & 0 & & \\ \hline
		& & 1 & \\
		& & & -1
	\end{array} \right) = \underbrace{T^{-1} H_{OCP} T}_{=:J} \quad\text{with}~T = \left( \begin{smallmatrix}
		1 & 0 & 1 & 1 \\
		0 & 1 & 1 & -1 \\
		0 & -1 & 0 & 0 \\
		0 & 0 & -1 & -1
	\end{smallmatrix} \right).
\end{equation}
Next, we compute $e^{H_{OCP}t} = T e^{Jt} T^{-1}$, which yields
\begin{eqnarray}\nonumber
	e^{At} & = & \left( \begin{array}{rrrr}
		1 & \sinh(t) & \sinh(t)-t & 1 - \cosh(t) \\
		0 & \cosh(t) & \cosh(t)-1 & -\sinh(t) \\
		0 & 0 & 1 & 0 \\
		0 & -\sinh(t) & -\sinh(t) & \cosh(t)
	\end{array} \right)
\end{eqnarray}
where we used the functions $\cosh(t) = \nicefrac{1}{2} (e^t+e^{-t})$ and $\sinh(t) = \nicefrac{1}{2} (e^t-e^{-t})$ to simplify the resulting expression. %
Thus, the solution of the state-adjoint system is given by
\begin{equation}\label{ExampleMotivationalSolution}
	\begin{pmatrix}
		q^\star(t) & v^\star(t) & \lambda_1(t) & \lambda_2(t)
	\end{pmatrix}^\top
	= e^{At} \begin{pmatrix}
		q_0 & v_0 & \lambda_1(0) & \lambda_2(0)
	\end{pmatrix}^\top.
\end{equation}
The third equation immediately implies $\lambda_1(t) = \lambda_1(0)$ for all $t \in [0,T]$ (and, in particular, $\lambda_1(T) = \lambda_1(0)$). 

In the following, we consider this system of equations for $t = T$ to make use of the boundary conditions in order to determine the unknowns $\lambda_1(0)$, $\lambda_2(0)$, and $\lambda_2(T)$. Adding the first to the fourth equation stated in~\eqref{ExampleMotivationalSolution} yields
\begin{equation}\label{ExampleMotivationalEq1}
	\lambda_2(T) = q_0 - q_T - T \cdot \lambda_1(0) + \lambda_2(0).
\end{equation}
Next, rearranging the second equation of~\eqref{ExampleMotivationalSolution} yields
\[
	\lambda_2(0) = \frac {\cosh(T) v_0 - v_T + (\cosh(T)-1) \lambda_1(0)}{\sinh(T)}.
\]
Plugging this expression for $\lambda_2(0)$ into the first equation leads to the equation
\begin{equation}\nonumber
	\lambda_1(0) = \frac{\sinh(T)(q_T-q_0)+(1-\cosh(T)) (v_0+v_T)}{2 (\cosh(T)-1) - T \sinh(T)} 
\end{equation}
Hence, using this expression yields $\lambda_2(0)$ and, consequently, allows for evaluating the velocity~$v^\star(t)$, $t \in (0,T]$.

\subsection{Hyperbolic Velocity Turnpike}

\begin{figure}
	\begin{center}
		\includegraphics[width=.75\textwidth]{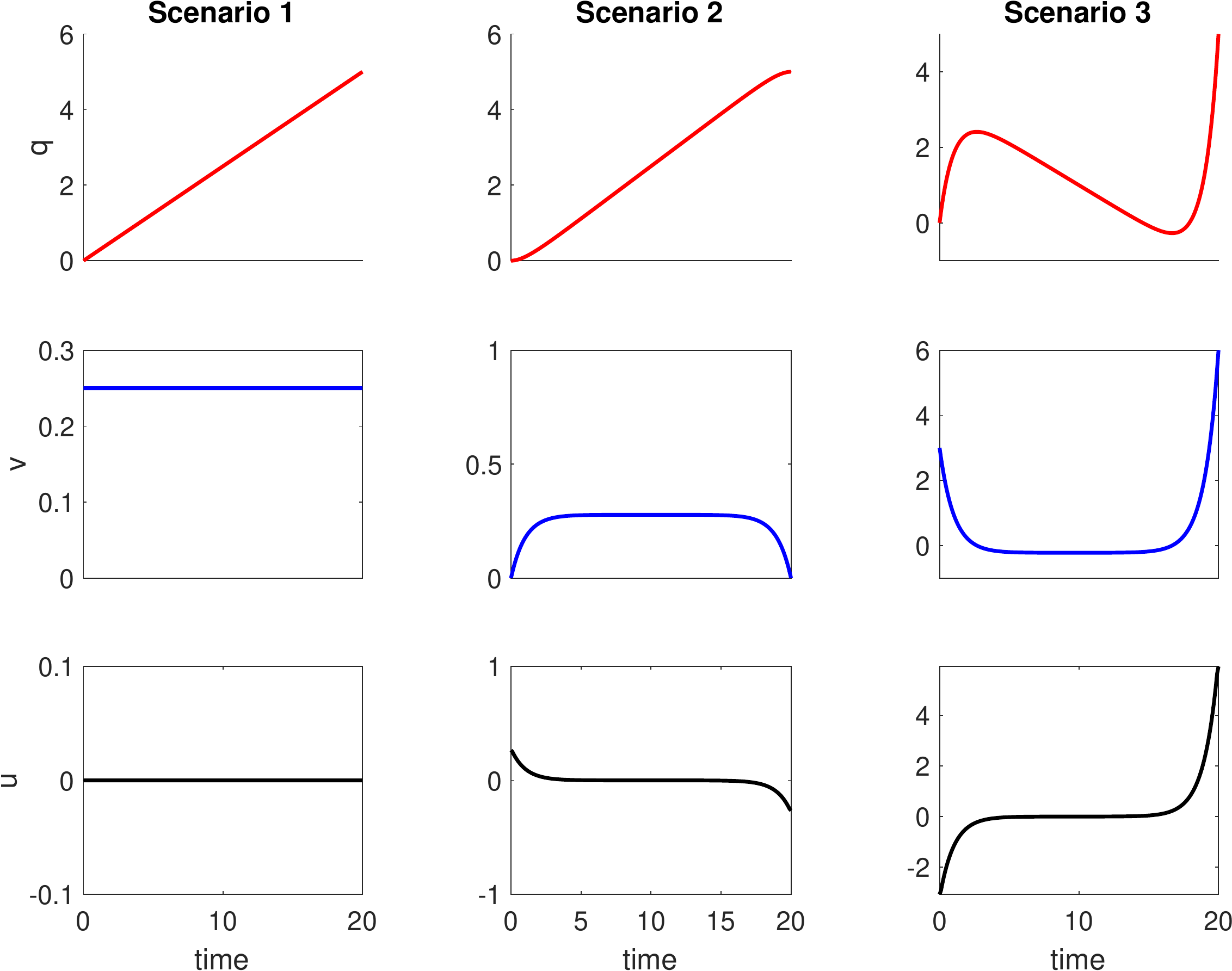}
		\caption{Numerical solution of the illustrative example  for $T = 20$.}
		\label{fig:motivational Example}
	\end{center}
\end{figure}

Here, we focus on the special case $v_0 = v_T = 0$. Moreover, we use the abbreviation $\tilde{q} := q_T - q_0$ since only the distance to traverse matters due to the translational invariance of the system in consideration. Then, we numerically demonstrate that also other boundary conditions essentially lead to the same result. 
\begin{prop}[Hyperbolic velocity turnpike]~\\
	For $v_0 = v_T = 0$ and all $q_0\in\mathbb{R}$ the optimal solutions $(q^\star(\cdot), v^\star(\cdot), u^\star(\cdot))$ exhibit an \textit{hyperbolic velocity turnpike} with respect to $(\bar v, \bar u) = (0,0)$, %
	i.e.\ there exists a positive constant $C$ such that, for all $q_0, q_T \in \mathbb{R}$ and all $T > 0$, 
	we have
	\begin{equation}\label{eq:vTP}
		\left\| (\bv^\star(t), \bu^\star(t)) - (\bar\bv, \bar\bu)\right\| \leq C / T
	\end{equation}
	for all $t \in [0,T]$. \eBox
\end{prop}
%
\begin{proof} For $v_0 = v_T = 0$, we get
\begin{align}
	\lambda_1(0) & = \left( \frac {\sinh(T)}{2(\cosh(T)-1) - T\sinh(T)}\right) \, \tilde{q},\nonumber \\
	\lambda_2(0) & = \left(\frac {\cosh(T)-1}{2(\cosh(T)-1) - T\sinh(T)}\right) \, \tilde{q}. \nonumber
\end{align}
and, thus, 
\begin{align}
	v^\star(t) & = (\cosh(t)-1) \lambda_1(0) - \sinh(t) \lambda_2(0) \nonumber \\
	& = \left(\frac {\sinh(t) + \sinh(T-t) - \sinh(T)}{2(\cosh(T)-1) - T\sinh(T)}\right) \, \tilde{q}. \nonumber
\end{align}
A direct calculation yields the first two derivatives of $v^\star(t)$:
\begin{align}
	v^{\star \prime}(t) = \left( \frac {\cosh(t) - \cosh(T-t)}{2(\cosh(T)-1) - T\sinh(T)} \right) \, \tilde{q}, \nonumber \\
	v^{\star \prime\prime}(t) = - \left( \frac {\sinh(T-t) + \sinh(t)}{2(\cosh(T)-1) - T\sinh(T)} \right) \, \tilde{q}. \nonumber
\end{align}
For $T>0$, the denominator is strictly negative. %
Hence, for $q_T > q_0$, the second derivative is negative definite and, thus, concave. %
Since the first derivative has its only zero at $t = T/2$, which is located in the interior of the domain $[0,T]$, and $v(0) = v(T) = 0$, we get 
\begin{align}\nonumber
	| v^\star(T/2) | = \max_{t \in [0,T]} | v(t) |
\end{align}
where the absolute value is only added to obtain the same assertion for $q_0 > q_T$ (by an analogous argumentation). %
Next, we show $| v^\star(T/2) | \leq C_{\tilde{q}} / T$ $\forall\, T > 0$ with $C_{\tilde{q}} = 3 |\tilde{q}| /2$:
\begin{eqnarray*}\nonumber
	\frac {T | v^\star(T/2) |}{|\tilde{q}|} & = & T \left( \frac{2 \sinh(T/2) (\cosh(T/2) - 1)}{T\sinh(T) - 2(\cosh(T)-1)} \right) \\
	& = & \frac{T (\sinh(T) - 2 \sinh(T/2))}{(T-2)\sinh(T) + 2 (1-e^{-T})} \\
	& = & \frac{ \sum_{k=2}^\infty \frac {T^{2k}}{(2k-1)!} \left( 1 - \frac 1 {2^{2(k-1)}} \right) }{ \sum_{k=2}^\infty \frac {T^{2k}}{(2k-1)!} \left( 1 - \frac 1k \right) } \leq \frac 32
\end{eqnarray*}
since $2 (1 - \frac 1 {2^{2(k-1)}}) \leq 3 (1-\frac 1k)$ holds for $k = 2$ with equality and for $k > 2$ with strict inequality. %
Clearly, if the set of initial positions is compact, $C_{\tilde{q}}$ could be uniformly estimated.
An (almost) analogous reasoning applies for the optimal control~$u^\star$ since $u^\star(t) = -\lambda_2(t)$: the nominator of $v^\star(t)$ is replaced by $\sinh(t)\sinh(T)-\cosh(t)(\cosh(T)-1)$. Hence, calculating $\frac {\mathrm{d}}{\mathrm{d}t} u^\star(t)$ directly shows that $u^\star$ is either strictly monotonically increasing ($\tilde{q} > 0$) or decreasing ($\tilde{q} < 0$). Consequently, the extrema are located at the boundaries. Here, we get 
\begin{equation}
	| u^\star(0) | = | u^\star(T) | = \frac {\cosh(T)-1}{T \sin(T) - 2(\cosh(T)-1)} |\tilde{q}|, \nonumber
\end{equation}
which can then be analogously estimated. Overall, $C$ is set to $\sqrt{C_{\tilde{q}}^2 + C_u^2}$.
\end{proof}

Note that the obtained result nicely fits to our intuition. If we double the available time~$T$, we may reduce the speed by $50\%$. %
The numerical approximations of the optimal solutions are obtained using the NLP solver WORHP, see \cite{Knauer16a}, and they are shown in Figure~\ref{fig:motivational Example}.
\begin{figure}
	\begin{center}
		\includegraphics[width=.55\textwidth]{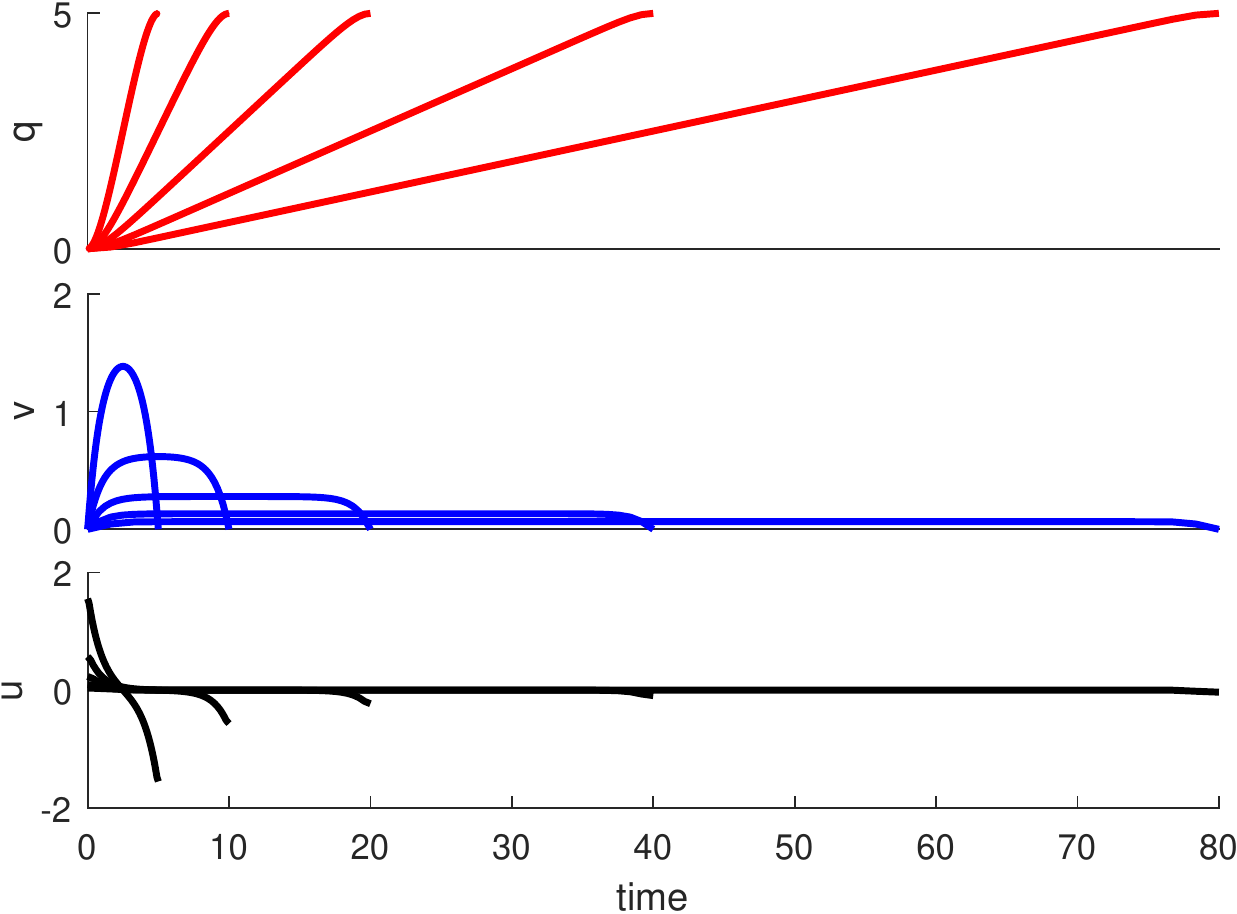}
		\caption{Numerical solution of the illustrative example  for $T \in  \{5,10,20,40,80\}$.}
		\label{fig:motivational ExampleII}
	\end{center}
\end{figure}
We consider $q_0 = 0$ and $q_T=5$ as boundary conditions on the configuration and a fixed final time $T=20$. %
If the boundary velocities are chosen to exactly match the average velocity which is needed for a distance of $\tilde{q} = 5$ in $\Delta t=20$ time steps, %
i.e.\ $v_0 = v_T = \frac{1}{4}$, the velocity turnpike is defined by $\bar v = \bar u =0$, while the optimal solution for $T = 20$ is given by the trim $v^\star = \frac{1}{4}$, cf. (Figure~\ref{fig:motivational Example}, left). %
In Figure~\ref{fig:motivational Example}, center, we give the solution for symmetric boundary values of the velocity, i.e $v_0 = v_T = 0$. 
Here, we observe the incoming arc and leaving arc of the optimal velocity. On the turnpike, $v$ is constant and $q$ increases again linearly. %
As a third scenario, let $v_0 = 3.0$ and $v_T = 6.0$. %
Again, the optimal solution has the predicted turnpike property at $\bar\bv = \bar\bu =0$ with zero control and thus constant velocity and linear decrease of configuration. %
Figure~\ref{fig:motivational ExampleII} shows the solutions for the boundary conditions $q_0 = 0$ and $q_T=5$, $v_0 = v_T = 0$ and 
 $T \in  \{5,10,20,40,80\}$. As expected the velocity turnpike occurs at $\bar\bv = \bar\bu =0$.

\section{Nonlinear Hovercraft Example}\label{sec:nonlinMech}

Now we turn towards a nonlinear example of a hovercraft. The system dynamics are governed by the second-order system
	\begin{align*}
		\begin{pmatrix}
			m \ddot{x} \\
			m \ddot{y} \\
			J \ddot{\theta}
		\end{pmatrix} = \begin{pmatrix}
			R_{\theta} \begin{pmatrix}
				u_1 \\
				u_2
			\end{pmatrix} \\
			-r u_2
		\end{pmatrix}		
	\end{align*}
	Observe that right hand side now depends on the rotation matrix $R_{\theta}$. For simplicity we assume mass and inertia to be equal to one, i.e.~$m=1$ and $J=1$.
	
We have the same behavior as in the previous example: The hovercraft is a second-order system and all accelerations vanish for $u_1=u_2\equiv 0$. 
In OCPs with stage cost $\ell(\mathbf{q}, \mathbf{v}, \mathbf{u})= v_x^2 + v_y^2 +v_\theta^2+ u_1^2+u_2^2$ and boundary conditions on $\mathbf{v}$ 	such that $\mathbf{q}^\star$ can be reached from $\mathbf{q}_0$ with constant $\mathbf{v} \equiv \mathbf{v}_0$, it turns out that indeed the optimal velocity is constant ${v}^\star_i \equiv \frac{q_i(T) -q_i(0)}{T}, i \in \{x,y,\theta\}$. %
Now we consider the parallel parking problem, i.e. $\mathbf{q}_0 = (0,1,0)^\top$ to $\mathbf{q}_T = (0,0,0)^\top$ with $\mathbf{v}_0 = \mathbf{v}_T = 0$. %
The optimal solution indeed seem to have a turnpike, cf. Figure~\ref{fig:hovercraftParking_01}. %
\begin{figure}
	\includegraphics[width=\textwidth]
	{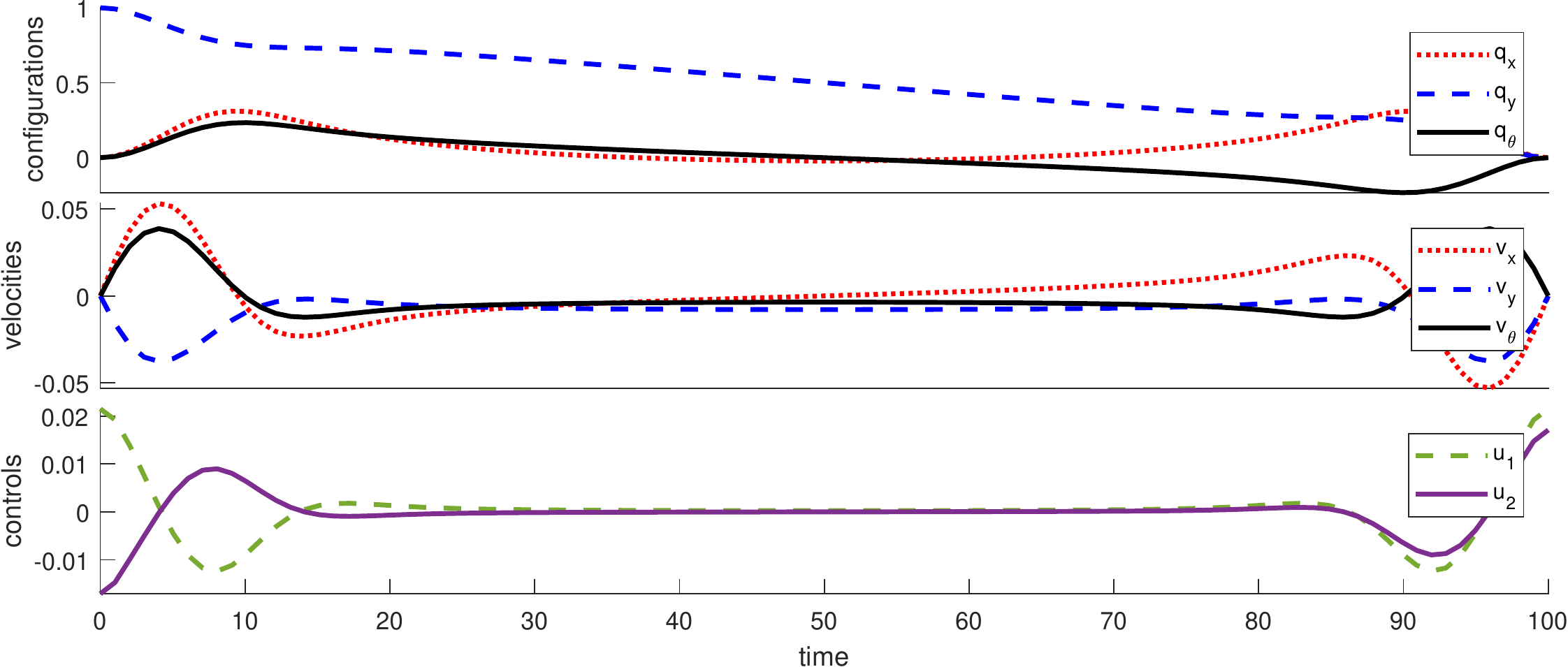}
	\caption{Hovercraft parallel parking example.} 
\label{fig:hovercraftParking_01}
\end{figure}

\section{Conclusions and Outlook} \label{sec:conclusions}
In this paper, we discussed time-varying turnpike properties in mechanical systems with symmetries. We proposed the concept of a velocity turnpike, which is a velocity steady state (or partial steady state). Specifically, we proposed to distinguish measure-based, exponential and hyperbolic velocity turnpikes. We have illustrated these concepts discussing two OCPs. 

Future work will investigated how dissipativity notions can be utilized to further analyze velocity turnpikes.

\bibliographystyle{abbrv}  

\bibliography{ReferencesTurnpikeMechSystem}             
                                                   







\end{document}